\documentclass{amsart}
\usepackage{amsfonts,amssymb,mathrsfs,amsthm}
\usepackage[colorlinks,linkcolor=black]{hyperref}
\usepackage{graphicx}
\usepackage[all]{xy}

\newtheorem{thm}{Theorem}[section]

\newtheorem{cor}[thm]{Corollary}
\newtheorem{lem}[thm]{Lemma}

\theoremstyle{definition}

\theoremstyle{remark}
\newtheorem*{rmk}{Remark}

\begin{document}

\title[Gauss sums of some matrix groups over $\Bbb Z/n\Bbb Z$
]{Gauss sums of some matrix groups over $\Bbb Z/n\Bbb Z$
}

\author{Su Hu}
\address{Department of Mathematics, South China University of Technology, Guangzhou 510640, China} \email{mahusu@scut.edu.cn}

\author{Guoxing He}\address{Department of Applied Mathematics, China Agricultural University, Beijing 100083, China}\email{guoxing\_{}he@cau.edu.cn}

\author{Yingtong Meng}\address{Department of Applied Mathematics, China Agricultural University, Beijing 100083, China}\email{Hallie@cau.edu.cn}

\author{Yan Li*}
\address{Department of Applied Mathematics, China Agricultural
University, Beijing 100083, China} \email{liyan\_00@cau.edu.cn}

 \thanks{*Corresponding author}
\subjclass[2010]{11A07, 11A25}
 \keywords
 {General linear group;
Special linear group; Residue class ring $\Bbb Z/n \Bbb Z$; Gauss sum; Ramanujan sum; Hyper-Kloosterman sum.}

\begin{abstract}In this paper, we will explicitly calculate Gauss sums for the general linear groups and the special linear groups over $\Bbb Z_n$, where $\Bbb Z_n=\Bbb Z/n \Bbb Z$ and $n>0$ is an integer. For $r$ being a positive integer, the formulae of Gauss sums for ${\rm GL}_r(\Bbb Z_n)$ can be expressed in terms of classical Gauss sums over $\Bbb Z_n$, while the formulae of Gauss sums for ${\rm SL}_r(\Bbb Z_n)$ can be expressed in terms of hyper-Kloosterman sums over $\Bbb Z_n$.
As an application, we count
the number of $r\times r$ invertible  matrices over $\Bbb Z_n$ with given trace by using the the formulae of Gauss sums for ${\rm GL}_r(\Bbb Z_n)$ and the orthogonality of Ramanujan sums.
\end{abstract}
\maketitle

\section{Introduction}

Classically, there are two kinds of Gauss sums: one is defined over $\Bbb{F}_{q}$, i.e., the finite field with $q$ elements, and the other one is defined over $\Bbb Z_n$, i.e., the residue class ring of $\Bbb Z$ modulo $n\Bbb Z$. For the case of $q=n=p$ being a prime, these two kinds of Gauss sums coincide, and are defined by
\begin{equation}\label{f1}\begin{aligned}G(\Bbb F_p, \chi,\lambda)
&=\sum_{a\in{\Bbb {F}}_{p}^{*}}\chi(a)\lambda(a),\end{aligned}\end{equation}
where $\chi$ is a multiplicative character of $\Bbb {F}_{p}^{*}$, i.e., the group of units of $\Bbb {F}_p$ and $\lambda$ is an additive character of $\Bbb {F}_p$. It is well known that Gauss initially used $G(\Bbb F_p, \chi,\lambda)$ with $\chi$ being a quadratic character, to prove the quadratic reciprocity law. For the general case, these two kinds of Gauss sums behave differently and need to be treated separately (see p.30 of \cite{Cohen}).

The Gauss sums are ubiquitous in number theory. For example, they naturally occur in the Stickelberger's theorem and the functional equations of Dirichlet $L$-functions. Therefore, the Gauss sums attract many
researcher's interests and have many generalizations.

The Gauss sums for classical groups over a finite field have been extensively studied by D.S.Kim et al. in a series of papers \cite{Kim1,Kim2,Kim3,Kim4,Kim5, Kim6, Kim7, Kim8}. 
  In \cite{Kim2}, Kim got the explict formulae of Gauss sums for the general linear group ${\rm GL}_r(\Bbb {F}_q)$ and the special linear group ${\rm SL}_r(\Bbb {F}_q)$ by using the Bruhat decomposition. Indeed, the Gauss sums for ${\rm GL}_r(\Bbb {F}_q)$ and ${\rm SL}_r(\Bbb {F}_q)$ are defined by
\begin{equation}\label{f3}\begin{aligned}G({\rm GL}_r(\Bbb {F}_q),\chi,\lambda)
&= \sum_{X\in{{\rm GL}_r({\Bbb {F}}_{q})}}\chi({\rm det}\ X)\lambda({\rm tr}\ X),\end{aligned}\end{equation}
\begin{equation}\label{f4}\begin{aligned}G({\rm SL}_r(\Bbb {F}_q),\lambda)
&= \sum_{X\in{{\rm SL}_r({\Bbb {F}}_{q})}}\lambda({\rm tr}\ X),\end{aligned}\end{equation}
where $\lambda$ is an additive character of $\Bbb {F}_q$ and $\chi$ is a multiplicative character of $\Bbb {F}_q^{*}$.
In \cite{Kim2}, Kim proved that for nontrivial $\lambda$,
\begin{equation}\label{f5}\begin{aligned}G({\rm GL}_r(\Bbb {F}_q),\chi,\lambda)
&=q^{r(r-1)/2}G(\Bbb {F}_q,\chi,\lambda)^r,\end{aligned}\end{equation}
\begin{equation}\label{f6}\begin{aligned}G({\rm SL}_r(\Bbb {F}_q),\lambda)
&=q^{r(r-1)/2}{K}_r(\Bbb {F}_q,\lambda),\end{aligned}\end{equation}
where 
 $G\left(\Bbb {F}_q,\chi,\lambda\right)$ is the classical Gauss sum over $\Bbb F_q$, defined similarly as (\ref{f1}), and
\begin{equation}\label{f7}\begin{aligned}K_r(\Bbb F_q,\lambda)
&=\sum_{\substack{x_{1}x_{2}\cdots x_{r}=1\\ x_{1},x_{2},\ldots, x_{r}\in\Bbb F_q}}\lambda(x_{1}+x_{2}+\cdots+x_{r}),\end{aligned}\end{equation}
is the hyper-Kloosterman sum over $\Bbb F_q$.

As Kim remarked, the formulae (\ref{f5}) and (\ref{f6}) already appeared in the work of Eichler \cite{Eichler} and Lamprecht \cite{Lamprecht} (see the introduction of \cite{Kim2}). Fulman \cite{Fulman} also got the formulae (\ref{f5}) by using the cycle index techniques.

 In \cite{LH}, Li and Hu studied the Gauss sums for ${\rm GL}_r(\Bbb F_q)$ and ${\rm SL}_r(\Bbb F_q)$ in a more general setting and got the explicit formulae for
\begin{equation}\label{f8}\begin{aligned}G({\rm GL}_r(\Bbb F_q),U,\chi,\lambda)=\sum_{X\in {\rm GL}_r(\Bbb F_q)}\chi({\rm det}\ X)\lambda(U\cdot X),\end{aligned}\end{equation}
and
\begin{equation}\label{f9}\begin{aligned}G({\rm SL}_r(\Bbb F_q),U,\lambda)=\sum_{X\in {\rm SL}_r(\Bbb F_q)}\lambda(U\cdot X),\end{aligned}\end{equation}
by a different method, where $U\in{{\rm M}_r(\Bbb F_q)}$ is an arbitrary $r\times r$ matrix over $\Bbb F_q$ and $U\cdot X={\rm tr}\ U^t X$ is the inner product of $U$ and $X$. Note that for a fixed non-trivial additive character $\lambda$, as $U$ varies over ${\rm M}_r(\Bbb F_q)$, $\lambda(U\cdot\ \ \ )$ runs over all additive characters of ${\rm M}_r(\Bbb F_q)$.
 As an application, they \cite{LH} calculated $N_\beta$, the number of matrices in ${\rm GL}_r(\Bbb F_q)$ with trace $\beta$, where $\beta\in\Bbb F_q$. Explicitly, we have
\begin{equation}\label{f10}\begin{aligned}
 \ \ \ \ \ N_\beta=&\left\{
\begin{array}{llll} \displaystyle q^{n(n-1)/2-1}\cdot\left(\prod_{i=1}^{n}(q^i-1)-(-1)^n\right) &\ &\mathrm{if}\ \beta\neq0\\
\displaystyle q^{n(n-1)/2-1}\cdot\left(\prod_{i=1}^{n}(q^i-1)+(-1)^n(q-1)\right) &\ &\mathrm{if}\ \beta=0
\end{array}\right.
\end{aligned}\end{equation}

Gauss sums (and more generally exponential sums) for matrix groups can be used to study the uniform distribution property for matrix groups (e.g., see \cite{Sh}, \cite{HL13}, \cite{Perret} and \cite{Shpar}).
Also they can  be applied in coding theory (see \cite{Kim9} and \cite{Ravagnani1,Ravagnani2}).

The aim of this paper is to generalize formulae (\ref{f5}), (\ref{f6}), (\ref{f10}) to general linear groups and special linear group over $\Bbb Z_n$.

From now on, we always assume that $\chi$ is a multiplicative character of $\Bbb Z_n^*$, i.e., the unit group of $\Bbb Z_n$, and $\lambda$ is an additive character of $\Bbb Z_n$. Explicitly, $\lambda$ can be uniquely written as
$$\lambda(x)={\rm exp}\left(\frac{2\pi\sqrt{-1}ax}{n}\right),\ {\rm where}\ x\in\Bbb Z_n,\ {\rm\ for\ some}\ a\in\Bbb Z_n.$$ The Gauss sums for the general linear group ${\rm GL}_r(\Bbb Z_n)$ and the special linear group ${\rm SL}_r(\Bbb Z_n)$ are defined as
\begin{equation}\label{f11}\begin{aligned}G({\rm GL}_r(\Bbb Z_n),\chi,\lambda)=\sum_{X\in {\rm GL}_r(\Bbb Z_n)}\chi({\rm det}\ X)\lambda({\rm tr}\ X),\end{aligned}\end{equation}
\begin{equation}\label{f12}\begin{aligned}G({\rm SL}_r(\Bbb Z_n),\lambda)=\sum_{X\in {\rm SL}_r(\Bbb Z_n)}\lambda({\rm tr}\ X),\end{aligned}\end{equation}
Note that for $r=1$, equation (\ref{f11}) is just the definition of classical Gauss sums for $\Bbb Z_n$, i.e.
\begin{equation}\label{f13}\begin{aligned}G(\Bbb Z_n,\chi,\lambda)=\sum_{X\in{\Bbb Z_n^*}}\chi(x)\lambda(x),\end{aligned}\end{equation}

In this paper, we will calculate the Gauss sums $G({\rm GL}_r(\Bbb Z_n),\chi,\lambda)$ and $G({\rm SL}_r(\Bbb Z_n),\lambda)$ explicitly (see Theorems \ref{Thm2}, \ref{Thm3} and \ref{Thm5} below). Note that the expression of $G({\rm SL}_r(\Bbb Z_n),\lambda)$ involves the hyper-Kloosterman sums, which are previously studied by a lot of researchers (e.g. see \cite{Luo}, \cite{Fisher}  \cite{Smith}, \cite{LiHZ} and \cite{Ye1, Ye2, Ye3}).

The key ingredient in calculation of (\ref{f11}) and (\ref{f12}) is averaging such sums over Borel subgroup, i.e. the subgroups consisting of upper triangular matrices. This method is similar to that of \cite{LH} and is originally inspired by \cite{Sh}.

Note that, Maeda \cite{Maeda} explicitly determine the following kind of Gauss sums for ${\rm GL}_2(\Bbb Z_{p^m})$:
\begin{equation}\label{T.Meada}\begin{aligned}G({\rm GL}_2(\Bbb Z_{p^m}),\Psi,\lambda)=\sum_{X\in {\rm GL}_2(\Bbb Z_{p^m})}\Psi(X)\lambda({\rm tr}\ X),\end{aligned}\end{equation}
where $\Psi$ is any irreducible character of group representations of ${\rm GL}_2(\Bbb Z_{p^m})$, $p$ is an odd prime and $m\geqslant2$ is an integer (The case of $m=1$ is contained in the results of \cite{Kondo} by Kondo, previously). In the abstract of \cite{Maeda}, Maeda remarked that ``While there are several studies of the Gauss sums on finite algebraic groups defined over a finite field, this paper seems to be the first one which determines the Gauss sums on a matrix group over a finite ring."

Finally, as an application, we count the number of matrices in ${\rm GL}_r(\Bbb Z_n)$ with trace $\beta$, where $\beta\in\Bbb Z_n$ (see Theorem \ref{Thm6} and Theorem \ref{Thm8}). The calculation relies on the orthogonality of Ramanujan sums (e.g. see \cite{Toth}), which can be viewed as a special kind of Gauss sums $G(\Bbb Z_n,\chi,\lambda)$ with trivial multiplicative character $\chi$.

\section{Preliminaries on Gauss sums over $\mathbb Z_n$}
Write the additive character
\begin{equation}\label{f14}\begin{aligned}\lambda(x)={\rm exp}\left(\frac{2\pi\sqrt{-1}ax}{n}\right),\quad {\rm with}\ d={\rm gcd}(a,n),\end{aligned}\end{equation}
for some $a\in\Bbb Z$, where $x\in\Bbb Z_n$ and ${\rm gcd}(\ ,\ )$ represents the greatest common divisor. Let $f$ be the conductor of the multiplicative character $\chi$, i.e. $f$ is the smallest positive integer such that $\chi$ factors through $\Bbb Z_f^*$.

First, we collect some results on Gauss sums over $\Bbb Z_n$ from Cohen's book (see p.31-33 and p.94-95 of \cite{Cohen}).
\begin{equation}\label{f15}\begin{aligned}
 \ \ \ \ \ G(\Bbb Z_n,\chi,\lambda)=&\left\{
\begin{array}{llll} \displaystyle 0 &\ &\mathrm{if}\ f\nmid n/d\\
\displaystyle \frac{\varphi(n)}{\varphi(n/d)}G(\Bbb Z_{n/d},\chi',\lambda) &\ &\mathrm{if}\ f\mid n/d
\end{array}\right.
,
\end{aligned}\end{equation}
where $\chi'$ is the character of $\Bbb Z_{n/d}^*$ induced by $\chi$, i.e.
\begin{equation}\label{f16}\begin{aligned}\chi'(c\ {\rm mod}\ n/d)=\chi(c\ {\rm mod}\ n),\quad{\rm for}\ {\rm gcd}(c,n)=1,\end{aligned}\end{equation}
and $\varphi$ is the Euler's totient function.

In equation (\ref{f15}), the case of $f\nmid n/d$ is just Proposition 2.1.40 of \cite{Cohen}. For $f\mid n/d$, both $\lambda$ and $\chi$ can be defined modulo $n/d$. We have
\begin{equation}\label{Cohen}\begin{aligned}\sum_{x\in{\Bbb Z_{n}^*}}\chi(x)\lambda(x)=\frac{\varphi(n)}{\varphi(n/d)}\sum_{y\in{\Bbb Z_{n/d}^*}}\chi'(y)\lambda(y)\end{aligned}\end{equation}
as $\varphi(n)/\varphi(n/d)$ many $x\in\Bbb Z_n^*$ maps to one $y\in\Bbb Z_{n/d}^*$. This yields the desired identity.

From equation (\ref{f15}), we only need to consider the case $d={\rm gcd}(a,n)=1$. Under this condition,
\begin{equation}\label{f17}\begin{aligned}G(\Bbb Z_n,\chi,\lambda)=\overline{\chi(a)}\sum_{x\in\Bbb Z_n^*}\chi(x){\rm exp}\left(\frac{2\pi \sqrt{-1}x}{n}\right),\end{aligned}\end{equation}
(see Proposition 2.1.39 of \cite{Cohen}).

Now, for $a=1$,
\begin{equation}\label{f18}\begin{aligned}G(\Bbb Z_n,\chi,\lambda)=\mu(\frac{n}{f})\chi_f(\frac{n}{f})G(\Bbb Z_f,\chi_f,\lambda_f),\end{aligned}\end{equation}where $\chi_f$ is the character of $\Bbb Z_f^*$ such that

{\centering\[\chi_f(c\ {\rm mod}\ f)=\chi_f(c\ {\rm mod}\ n)\quad {\rm for}\ {\rm gcd}(c,n)=1;\]}
{\centering\[\lambda_f(x)={\rm exp}(\frac{2\pi ix}{f})\quad {\rm for}\ x\in\Bbb Z_f;\]}

\leftline{and $\mu$ is the M$\rm\ddot{o}$bius function (see exercise 12 on page 95 of \cite{Cohen}).}

It is also known that
\begin{equation}\label{f19}\begin{aligned}|G(\Bbb Z_f,\chi_f,\lambda_f)|=f^\frac{1}{2}\end{aligned}\end{equation}
(see Proposition 2.1.45 of \cite{Cohen}).

\section{Main results}
\begin{thm}\label{Thm1}Assume the order of $\lambda$ is $n$, i.e. $d={\rm gcd}(a,n)=1$ in equation \eqref{f14}. Then
\begin{equation}\label{f20}\begin{aligned}G({\rm GL}_r(\Bbb Z_n),\chi,\lambda)=G(\Bbb Z_n,\chi,\lambda)^r n^{\frac{r(r-1)}{2}}.\end{aligned}\end{equation}
\end{thm}
\begin{proof}
Let ${\rm B}_r(\Bbb Z_n)$ be the Borel subgroup of ${\rm GL}_r(\Bbb Z_n)$, i.e., the group of upper triangular invertible matrices over $\Bbb Z_n$.
Averaging equation (\ref{f11}) over ${\rm B}_r(\Bbb Z_n)$ and changing the order of summation, we get
\begin{equation}\label{f20(1)}\begin{aligned}&\ \ \sum_{X\in {\rm GL}_r(\Bbb Z_n)}\chi({\rm det}\ X)\lambda({\rm tr}\ X)\\=&\frac{1}{\varphi(n)^rn^{\frac{r(r-1)}{2}}}\sum_{B\in {\rm B}_r(\Bbb Z_n)}\sum_{X\in {\rm GL}_r(\Bbb Z_n)}\chi({\rm det}\ BX)\lambda({\rm tr}\ BX)\\=&\frac{1}{\varphi(n)^rn^{\frac{r(r-1)}{2}}}\sum_{X\in {\rm GL}_r(\Bbb Z_n)}\chi({\rm det}\ X)\sum_{B\in {\rm B}_r(\Bbb Z_n)}\chi({\rm det}\ B)\lambda({\rm tr}\ BX).\end{aligned}\end{equation}

Writing
\begin{equation}\label{f20(2)}\begin{aligned}B=(b_{ij})\in {\rm B}_r(\Bbb Z_n), X=(x_{ij})\in {\rm GL}_r(\Bbb Z_n)\end{aligned}\end{equation}
and using the multiplicative property of characters, we get

\begin{equation}\label{f21}\begin{aligned}&\ \ \frac{1}{\varphi(n)^r n^{\frac{r(r-1)}{2}}}\sum_{X\in {\rm GL}_r(\Bbb Z_n)}\chi({\rm det}\ X)\sum_{(b_{ij})\in {\rm B}_r(\Bbb Z_n)}\prod\limits_{i=1}^{r}\chi(b_{ii})\lambda(b_{ii}x_{ii})\prod\limits_{i<j}\lambda(b_{ij}x_{ji})\\=& \frac{1}{\varphi(n)^r n^{\frac{r(r-1)}{2}}}\sum_{X\in {\rm GL}_r(\Bbb Z_n)}\chi({\rm det}\ X)\prod\limits_{i=1}^{r}\sum_{b_{ii}\in{\Bbb Z_n^*}}\chi(b_{ii})\lambda(b_{ii}x_{ii})\prod\limits_{i<j}\sum_{b_{ij}\in\Bbb Z_n}\lambda(b_{ij}x_{ji})\end{aligned}\end{equation}

By the assumption ${\rm gcd}(a,n)=1$, we know that $\lambda(x_{ji}\cdot\ \ \ )$ is a nontrivial character of $\Bbb Z_n$ if $x_{ji}\neq 0$ in $\Bbb Z_n$, where $1\leqslant i<j\leqslant r$. Therefore, we have
\begin{equation}\label{f22}\begin{aligned}
\ \ \ \ \ \sum_{b_{ij}\in{\Bbb Z_n}}\lambda(b_{ij}x_{ji})=&\left\{
\begin{array}{llll}\displaystyle n &\ &\mathrm{if}\ x_{ji}=0\\
\displaystyle 0 &\ &\mathrm{otherwise}
\end{array}\right.
,
\end{aligned}\end{equation}
by the orthogonality of characters, where $1\leqslant i<j\leqslant r$.

From (\ref{f22}), the summation item for $X\in {\rm GL}_r(\Bbb Z_n)$ in \eqref{f21} is nonzero only when $X\in {\rm B}_r(\Bbb Z_n)$ happens. Then, substituting (\ref{f22}) into (\ref{f21}), we obtain
\begin{equation}\label{f22(1)}
\begin{split}
&G({\rm Gl}_r(\mathbb Z_n),\chi,\lambda)\\
=&\frac{1}{\varphi(n)^rn^{\frac{r(r-1)}{2}}}\sum_{X\in {\rm B}_r(\mathbb{Z}_n)}\chi({\rm det}\ X)\prod^r_{i=1}\sum_{b_{ii}\in\mathbb{Z}_n^*}\chi(b_{ii})\lambda(b_{ii}x_{ii})\cdot n^{\frac{r(r-1)}{2}}\\
=&\frac{1}{\varphi(n)^r}\sum_{X\in {\rm B}_r(\mathbb{Z}_n)}\prod^r_{i=1}\sum_{b_{ii}\in\mathbb{Z}_n^*}\chi(b_{ii}x_{ii})\lambda(b_{ii}x_{ii})\\
=&G(\mathbb{Z}_n,\chi,\lambda)^rn^{\frac{r(r-1)}{2}}
\end{split}
\end{equation}
\end{proof}
To treat the general case $d={\rm gcd}(a,n)\neq1$, we need the following lemmas. Let $n=p_1^{m_1}\cdots p_s^{m_s}$ be the prime factorization of $n$. Then by the Chinese Remainder Theorem, it can be seen that
\begin{equation}\label{f23}
{\rm GL}_r(\mathbb{Z}_n)\simeq {\rm GL}_r(\mathbb{Z}_{p_{1}^{m_1}})\times {\rm GL}_r(\mathbb{Z}_{p_{2}^{m_2}})\times \cdots \times {\rm GL}_r(\mathbb{Z}_{p_{s}^{m_s}}).
\end{equation}
\begin{lem}\label{lemma1} Assume $p$ is a prime and $m$ is a positive integer. Then the cardinality of ${\rm GL}_r(\Bbb Z_{p^m})$ is
\begin{equation}\label{f24}
|{\rm GL}_r(\mathbb{Z}_{p^m})|=p^{mr^2}\prod_{i=1}^r(1-p^{-i}).
\end{equation}
\end{lem}
\begin{proof}See  Proposition 11.15 of \cite{Rosen} on page 182.
\end{proof}
\begin{lem}\label{lemma2}The cardinality of ${\rm GL}_r(\Bbb Z_n)$ is
\begin{equation}\label{f25}
|{\rm GL}_r(\mathbb{Z}_n)|=n^{r^2}\prod_{p\mid n}\prod^r_{i=1}(1-p^{-i}).
\end{equation}
\end{lem}
\begin{proof} Combine (\ref{f23}) and (\ref{f24}) together.
\end{proof}
\begin{thm}\label{Thm2} Let the order of $\lambda$ be $n/d$, i.e. $d={\rm gcd}(a,n)$ in \eqref{f14}.
Let $f$ be the conductor of $\chi$. Then, for $f|n/d$, we have
\begin{equation}\label{f25(1)}
G({\rm GL}_r(\mathbb{Z}_n),\chi,\lambda)=n^{\frac{r(r-1)}{2}}d^{\frac{r(r+1)}{2}}
\prod_{\mbox{\tiny$\begin{array}{c}
p\mid n\\
p\nmid n/d
\end{array}$}}\prod^r_{i=1}(1-p^{-i})G(\mathbb{Z}_{n/d},\chi',\lambda)^r
\end{equation}
where $\chi'$ is the character of $\Bbb Z_{n/d}^*$ satisfying \eqref{f16}. Otherwise, for $f\nmid n/d$, we have $G({\rm GL}_r(\Bbb Z_n),\chi,\lambda)=0.$
\end{thm}
\begin{proof} Similarly as in the proof of Theorem \ref{Thm1}, we have $G({\rm GL}_r(\Bbb Z_n),\chi,\lambda)$ equals to the right hand side of (\ref{f21}), i.e.
\begin{equation}\label{LiAdd1}\frac{1}{\varphi(n)^r n^{\frac{r(r-1)}{2}}}\sum_{X\in {\rm GL}_r(\Bbb Z_n)}\chi({\rm det}\ X)\prod\limits_{i=1}^{r}\sum_{b_{ii}\in{\Bbb Z_n^*}}\chi(b_{ii})\lambda(b_{ii}x_{ii})\prod\limits_{i<j}\sum_{b_{ij}\in\Bbb Z_n}\lambda(b_{ij}x_{ji})\end{equation}
 If $f \nmid n/d$, then by equation (\ref{f15}), the Gauss sums
\begin{equation}\label{f26}
\sum_{b_{ii}\in\mathbb{Z}_n^*}\chi(b_{ii})\lambda(b_{ii}x_{ii})=0,
\end{equation}
for $1\leqslant i\leqslant r$. Combining \eqref{LiAdd1} and (\ref{f26}), we get $G({\rm GL}_r(\Bbb Z_n),\chi,\lambda)=0$, for $f\nmid n/d$.

Now assume $f|n/d$. Both $\chi$ and $\lambda$ are defined modulo $n/d$. Consider the natural homomorphism,
\begin{equation}\label{f26(1)}
{\rm GL}_r(\mathbb{Z}_n)\xrightarrow{\pi}{\rm GL}_r(\mathbb{Z}_{n/d}):~X\mapsto X\mod n/d,~\mathrm{for}~X\in {\rm GL}_r(\mathbb{Z}_n)
\end{equation}
which is clearly surjective and whose kernel has the cardinality:
\begin{equation}\label{f27}
\begin{split}
|\ker\pi|&=|{\rm GL}_r(\mathbb{Z}_n)|/|{\rm GL}_r(\mathbb{Z}_{n/d})|\\
&=d^{r^2}\prod_{\mbox{\tiny$\begin{array}{c}
p\mid n\\
p\nmid n/d
\end{array}$}}\prod^r_{i=1}(1-p^{-i})
\end{split}
\end{equation}
by Lemma \ref{lemma2}.

Clearly, if $\pi(X)=\pi(X')$, i.e. $X\equiv X'\pmod{n/d}$, then $${\rm det}\ X\equiv {\rm det}\ X'\pmod{n/d},~\mathrm{and}~{\rm tr}\ X\equiv {\rm tr}\ X'\pmod{n/d}.$$
Therefore,
\begin{equation}\label{f28}
\begin{split}
\chi({\rm det}\ X)\lambda({\rm tr}\ X)&=\chi({\rm det}\ X')\lambda({\rm tr}\ X')=\chi'({\rm det}~ \pi(X))\lambda(\mathrm{tr}~\pi(X)).
\end{split}
\end{equation}

Combining (\ref{f11}), (\ref{f27}) and (\ref{f28}), we get

\begin{equation}\label{f29}
G({\rm GL}_r(\mathbb{Z}_n),\chi,\lambda)
=d^{r^2}
\prod_{\mbox{\tiny$\begin{array}{c}
p\mid n\\
p\nmid n/d
\end{array}$}}\prod^r_{i=1}(1-p^{-i})G({\rm GL}_r(\mathbb{Z}_{n/d}),\chi',\lambda).
\end{equation}

Substituting (\ref{f20}) into (\ref{f29}), we get

\begin{equation}\label{f30}
G({\rm GL}_r(\mathbb{Z}_n),\chi,\lambda)
=n^{\frac{r(r-1)}{2}}d^{\frac{r(r+1)}{2}}
\prod_{\mbox{\tiny$\begin{array}{c}
p\mid n\\
p\nmid n/d
\end{array}$}}\prod^r_{i=1}(1-p^{-i})G(\mathbb{Z}_{n/d},\chi',\lambda)^r
\end{equation}
which concludes the proof.
\end{proof}
Combining equation \eqref{f15} and Theorem \ref{Thm2}, we get the following uniform result.

\begin{thm}\label{Thm3} Let the order of $\lambda$ be $n/d$, i.e. $d={\rm gcd}(a,n)$ in \eqref{f14}. Then
\begin{equation}\label{f31}
\begin{split}
&G({\rm GL}_r(\mathbb{Z}_n),\chi,\lambda)\\=&n^{\frac{r(r-1)}{2}}d^{\frac{r(r-1)}{2}}
\prod_{\mbox{\tiny$\begin{array}{c}
p\mid n\\
p\nmid n/d
\end{array}$}}\prod^r_{i=1}(1-p^{-i})/(1-p^{-1}) G(\mathbb{Z}_{n},\chi,\lambda)^r.
\end{split}
\end{equation}
\end{thm}
\begin{proof} Let $f$ be the conductor of $\chi$. If $f\nmid n/d$, then the both sides of (\ref{f31}) are zero by Theorem \ref{Thm2}. Otherwise, by (\ref{f15}), we have
\begin{equation}\label{f32}
G(\mathbb{Z}_n,\chi,\lambda)=d\prod_{\mbox{\tiny$\begin{array}{c}
p\mid n\\
p\nmid n/d
\end{array}$}}(1-\frac{1}{p})G(\mathbb{Z}_{n/d},\chi',\lambda)
\end{equation}
Substituting (\ref{f32}) into \eqref{f25(1)}, we get the desired result.
\end{proof}
\begin{cor} Let $f$ be the conductor of $\chi$ and $n/d$ be the order of $\lambda$. Then
\begin{equation}\label{f33}
|G({\rm GL}_r(\mathbb{Z}_{n}),x,\lambda)|=n^{\frac{r(r-1)}{2}}d^{\frac{r(r+1)}{2}}f^{\frac{r}{2}}
\prod_{\mbox{\tiny$\begin{array}{c}
p\mid n\\
p\nmid n/d
\end{array}$}}\prod^r_{i=1}(1-p^{-i})
\end{equation}
if $f|n/d$; $n/(df)$ is squarefree and coprime to $f$. Otherwise $G({\rm GL}_r(\Bbb Z_n),\chi,\lambda)=0$.
\end{cor}
\begin{proof} Combining \eqref{f25(1)}, (\ref{f17}), (\ref{f18}) and (\ref{f19}), we get the desired result.
\end{proof}
Next, we consider the case of special linear groups.

\begin{thm} Assume the order of $\lambda$ is $n$, i.e. $d={\rm gcd}(a,n)=1$ in \eqref{f14}. Then
\begin{equation}\label{f34}
G({\rm SL}_r(\mathbb{Z}_{n}),\lambda)=n^{\frac{r(r-1)}{2}}K_r(\mathbb{Z}_{n},\lambda),
\end{equation}
where
\begin{equation}\label{f35}
K_r(\mathbb{Z}_{n},\lambda)=\sum_{\mbox{\tiny$\begin{array}{c}
x_1\cdots x_r=1\\
x_1, \ldots,x_r\in \mathbb{Z}_n
\end{array}$}}\lambda(x_1+x_2+ \cdots+x_r)
\end{equation}
is the hyper-Kloosterman sum over $\mathbb{Z}_{n}$.
\end{thm}
\begin{proof} Let $\widetilde {\rm B}_r(\Bbb Z_n)$ be the Borel subgroup of ${\rm SL}_r(\Bbb Z_n)$, i.e., the group of upper triangular matrices with determinate 1 over $\Bbb Z_n$. Averaging equation \eqref{f12} over $\widetilde {\rm B}_r(\Bbb Z_n)$ and changing the order of summation, we get
\begin{equation}\label{f36}
\sum_{X\in {\rm SL}_r(\mathbb{Z}_{n})}\lambda(\mathrm{tr}~X)=\frac{1}{\varphi(n)^{r-1} n^{\frac{r(r-1)}{2}}}\sum_{X\in {\rm SL}_r(\mathbb{Z}_{n})}\sum_{B\in \widetilde {\rm B}_r(\mathbb{Z}_n)}\lambda(\mathrm{tr}~BX)
\end{equation}

Substituting $X=(x_{ij})$ and $B=(b_{ij})$ into \eqref{f36}
and using the multiplicative property of $\lambda$, we obtain
\begin{equation}\label{f37}\begin{split}
&\frac{1}{\varphi(n)^{r-1} n^{\frac{r(r-1)}{2}}}\sum_{X\in {\rm SL}_r(\mathbb{Z}_{n})}\sum_{B\in \widetilde {\rm B}_r(\mathbb{Z}_n)}\prod_{i\leqslant j}\lambda(b_{ij}x_{ji})\\
=&\frac{1}{\varphi(n)^{r-1} n^{\frac{r(r-1)}{2}}}\sum_{X\in{{\rm SL}_r(\mathbb{Z}_{n})}}\sum_{b_{11}\cdots b_{rr}=1}\lambda(\sum_{i=1}^rb_{ii}x_{ii})\cdot\prod_{i<j}\sum_{b_{ij}\in{\mathbb{Z}_{n}}}\lambda(b_{ij}x_{ji})
\end{split}\end{equation}

By the assumption $d={\rm gcd}(a,n)=1$, we have

\begin{equation}\label{f38}
\sum_{b_{ij}\in\mathbb{Z}_n}\lambda(b_{ij}x_{ji})=\begin{cases}
n,&\text{$x_{ji}=0$},\\
0,&
\text{$x_{ji}\neq0$}.
\end{cases}
\end{equation}
Therefore, the summation item for $X\in{{\rm SL}_r(\Bbb Z_n)}$ in (\ref{f37}) is nonzero only if $X\in{\widetilde {\rm B}_r(\Bbb Z_n)}$. Substituting (\ref{f38}) into (\ref{f37}), we get

\begin{equation}\label{f38(1)}
G({\rm SL}_r(\Bbb Z_n),\lambda)=\frac{1}{\varphi(n)^{r-1}}\sum_{X\in \widetilde {\rm B}_r(\mathbb{Z}_n)}\sum_{b_{11}\cdots b_{rr}=1}\lambda(b_{11}x_{11}+ \cdots + b_{rr}x_{rr}).
\end{equation}
As $x_{11}x_{22}\cdots x_{rr}=1$, we get the desired identity (\ref{f34}).
\end{proof}
\begin{thm} \label{Thm5}Let $n/d$ be the order of $\lambda$. Then

\begin{equation}\label{f39}
G({\rm SL}_r(\mathbb{Z}_{n}),\lambda)=n^{r(r-1)/2}d^{r(r+1)/2-1}\prod_{\mbox{\tiny$\begin{array}{c}
p\mid n\\
p\nmid n/d
\end{array}$}}\prod^r_{i=2}(1-p^{-i})\cdot K_r(\mathbb{Z}_{n/d},\lambda)
\end{equation}
where $K_r(\mathbb{Z}_{n/d},\lambda)$ is the hyper-Kloosterman sum over $\Bbb Z_{n/d}$ defined in \eqref{f35}.
\end{thm}
\begin{proof} Consider the natural homomorphism
\begin{equation}\label{f39(1)}
{\rm SL}_r(\mathbb{Z}_n)\xrightarrow{\pi}{\rm SL}_r(\mathbb{Z}_{n/d}),~X\mapsto X\ {\rm mod}\ n/d,~\mathrm{for}~X\in {\rm SL}_r(\mathbb{Z}_n)
\end{equation}
By a result of Shimura (see the proof of Lemma 1.38 of \cite{Shimura} on page 21), we know that $\pi$ is surjective.

 Clearly
\begin{equation}\label{f39(2)}
|{\rm SL}_r(\mathbb{Z}_n)|=|{\rm GL}_r(\mathbb{Z}_n)|/|\mathbb{Z}_n^*|.
\end{equation}
Therefore, by Lemma 2, we have
\begin{equation}\label{f40}
|{\rm SL}_r(\mathbb{Z}_n)|=n^{r^2-1}\prod_{p\mid n}\prod^r_{i=2}(1-p^{-i}).
\end{equation}

Combining the subjectivity of $\pi$ and \eqref{f40}, we get
\begin{equation}\label{f41}
|\ker \pi|=d^{r^2-1}\prod_{\mbox{\tiny$\begin{array}{c}
p\mid n\\
p\nmid n/d
\end{array}$}}\prod^r_{i=2}(1-p^{-i})
\end{equation}

Note that for $X, X'\in{{\rm SL}_r(\Bbb Z_n)}$, if $\pi(X)=\pi(X')$, i.e. $X\equiv X'\pmod{n/d}$, then ${\rm tr}\ X\equiv {\rm tr}\ X' \pmod{ n/d}$, which implies $$\lambda({\rm tr}\ X)=\lambda({\rm tr}\ X')=\lambda({\rm tr}\ \pi(X)).$$ Therefore,
\begin{equation}\label{f42}
G({\rm SL}_r(\mathbb{Z}_{n}),\lambda)=|\ker \pi|\cdot G({\rm SL}_r(\mathbb{Z}_{n/d}),\lambda).
\end{equation}

Substituting (\ref{f34}) and (\ref{f41}) into (\ref{f42}), we get the desired identity.
\end{proof}
\begin{rmk}
In \cite{Smith} and \cite{Fisher}, some bounds of hyper-Kloosterman sums $K_r(\Bbb Z_n,\lambda)$ are given. Using their results, we can derive bounds of Gauss sums $G({\rm SL}_r(\mathbb{Z}_{n}),\lambda)$ by Theorem \ref{Thm5}.

For convinience of readers, here we rewrite their results by our notations.

Let $n=p_1^{m_1}p_2^{m_2}\cdots p_s^{m_s}$ be the prime factorization of $n$. Denote $p_i^{m_i}$ by $q_i$, for $1\leqslant i\leqslant s$. Let $w_1,w_2,\cdots,w_s\in{\Bbb Z}$ satisfy

\begin{equation}\label{f43}
w_1\frac{n}{q_1}+w_2\frac{n}{q_2}+\cdots+w_s\frac{n}{q_s}=1
\end{equation}
and denote $w_i n/q_i$ by $e_i$, where $1\leqslant i\leqslant s$. By Theorem 1 of \cite{Smith}, we have
\begin{equation}\label{f44}
K_r(\mathbb{Z}_{n},\lambda)=K_r(\mathbb{Z}_{q_1},\lambda_{e_1})\cdots K_r(\mathbb{Z}_{q_s},\lambda_{e_s})
\end{equation}
where $\lambda_{e_i}$ is a character of $\Bbb Z_{q_i}$ such that
\begin{equation}\label{f45}
\lambda_{e_i}(x)=\lambda(e_ix),~\forall\ x\in\mathbb{Z}_{q_i}
\end{equation}
where $1\leqslant i\leqslant s$.

Now assume the order of $\lambda$ is $n$, i.e., $d={\rm gcd}(a,n)=1$

Then by Theorem 6 of \cite{Smith}, we have
\begin{equation}\label{f46}
|K_r(\mathbb{Z}_{n},\lambda)|\leqslant n^{\frac{r-1}{2}}d_r(n)
\end{equation}
where $d_r(n)$ denotes the number of representations of $n$ as a product of $r$ factors. Explicitly,
\begin{equation}\label{f47}
d_r(n)=\prod_{p^m\|n}
\begin{pmatrix} m+r-1 \\ m \end{pmatrix}
\end{equation}

Next, we will state the bounds of \cite{Fisher}.

By (\ref{f44}), to bound $K_r(\Bbb Z_n,\lambda)$, it suffices to consider the case $n=p^m$ being a prime power. We still assume the order of $\lambda$ is $n$, which is sufficient for our purpose.

Let $h=v_p(r)$ be the exact power of $p$ dividing $r$ and $\lfloor x\rfloor$ be the floor function, i.e. the greatest integer $\leqslant x$. From Example 1.17 of \cite{Fisher}, we know that
\begin{equation*}\label{f48}
\begin{split}
&|K_r(\mathbb{Z}_{p^m},\lambda)|\\
\leqslant&\begin{cases}
rp^{\frac{(r-1)}{2}},&\text{$m=1$};\\
rp^{\frac{m(r-1)}{2}},&\text{$h=0$};\\
p^{v_p(2)-h/2}rp^{\frac{m(r-1)}{2}},&\text{$h>0$~and~$m\geqslant3h+2+4v_p(2)$};\\
p^{v_p(2)+\min\{h,\lfloor m/2\rfloor-1-v_p(2)\}}{\rm gcd}(r,p-1)p^{\frac{m(r-1)}{2}},&\text{otherwise}.
\end{cases}
\end{split}
\end{equation*}

In the above inequality, the case of $m=1$ is originally due to Deligne \cite{Deligne}.
\end{rmk}

\section{Applications}

For $\beta\in\mathbb{Z}_{n}$, let
\begin{equation}\label{f49}
N_{\beta}=\left|\{x\in {\rm GL}_r(\mathbb{Z}_{n})\big|{\rm tr}\ X=\beta\}\right|
\end{equation}
In this section, we will calcute $N_{\beta}$ by applying Theorem \ref{Thm3}.

Note that for $c\in{\Bbb Z_n^*}$, $N_{c\beta}=N_{\beta}$ since multiplication by $c$ is a bijection of ${\rm GL}_r(\Bbb Z_n)$. Therefore,
\begin{equation}\label{f50}
N_{\beta}=N_{{\rm gcd}(\beta,n)}.
\end{equation}
and only $N_l$ for $l|n$ needs to be computed.

For our purpose, we also need some knowledge of Ramanujan sums (see p.1-2 of \cite{Toth}). By definition, the Ramanujan sum $C_n(k)$ is the sum of $k$-th powers of $n$-th primitive roots of unity ($k\in{\Bbb Z}$), i.e.
\begin{equation}\label{f51}
C_n(k)=\sum_{j\in{\Bbb Z_n^*}}{\rm exp}\left(\frac{j k 2\pi\sqrt{-1}}{n}\right).
\end{equation}

Clearly, Ramanujan sums are Gauss sums with trivial multiplicative characters $\chi=1$, i.e.
\begin{equation}\label{f52}
C_n(k)=G(\Bbb Z_n, 1,\lambda_{k}),
\end{equation}
where $\lambda_{k}$ is the character of $\Bbb Z_n$ such that
\begin{equation}\label{f53}
\lambda_k(x)={\rm exp}\left(\frac{kx 2\pi\sqrt{-1}}{n}\right),~\forall\ x\in\mathbb{Z}_{n}
\end{equation}

It can also be expressed as
\begin{equation}\label{f54}
C_n(k)=\sum_{d|{\rm gcd}(k,n)}d\mu(n/d).
\end{equation}

The Ramanujan sums enjoy the following orthogonal properties:
\begin{equation}\label{f55}
\frac{1}{{\rm lcm}(l,n)}\sum^{{\rm lcm}(l,n)}_{k=1}C_l(k)C_n(k)=\begin{cases}
\varphi(n),&\text{if~$l=n$},\\
0,&\text{otherwise}.
\end{cases}
\end{equation}
where ${\rm lcm}$ represents the least common multiple.

\begin{thm}\label{Thm6} For $\beta\in\Bbb Z_n$, let $l={\rm gcd}(\beta,n)$. Then,
\begin{equation}\label{f55(1)}
\begin{split}
N_{\beta}=\frac{1}{\varphi(n/l)}&n^{r(r-1)/2-1}\sum_{d\mid n}C_{n/l}(d)C_{n}^r(d)\varphi(n/d)\\ &\times d^{r(r-1)/2}\prod_{\mbox{\tiny$\begin{array}{c}
p\mid n\\
p\nmid n/d
\end{array}$}}\prod^r_{i=1}(1-p^{-i})/(1-p^{-1})
\end{split}
\end{equation}
\end{thm}
\begin{proof} By definition, $G({\rm GL}_r(\Bbb Z_n),1,\lambda_k)$ equals to

\begin{equation}\label{f56}
\begin{split}
&\sum_{X\in {\rm GL}_r(\mathbb{Z}_{n})}\lambda_k(\mathrm{tr}~X)\\
=&\sum_{l\mid n}\sum_{\gcd({\rm tr}~X,n)=l}\lambda_k(\mathrm{tr}~X)\\
=&\sum_{l\mid n}N_l\times\sum_{\mbox{\tiny$\begin{array}{c}j=1\\\gcd(j,n/l)=1\end{array}$}}^{n/l}{\rm exp}\left(\frac{kjl2\pi\sqrt{-1}}{n}\right).
\end{split}
\end{equation}
The last equality is by using (\ref{f50}), replacing ${\rm tr}\ X=jl$, and using (\ref{f53}).
Substituting (\ref{f51}) into (\ref{f56}) with $n$ by $n/l$, we get
\begin{equation}\label{f57}
\sum_{l\mid n}N_lC_{n/l}(k)=G({\rm GL}_r(\mathbb{Z}_{n}),1,\lambda_k).
\end{equation}
Substituting Theorem \ref{Thm3} into (\ref{f57}) and using \eqref{f52}, we get
\begin{equation}\label{f58}
\sum_{l\mid n}N_lC_{n/l}(k)=n^{\frac{r(r-1)}{2}}C_n^r(k)A_{n}(k)
\end{equation}
where
\begin{equation}\label{f59}
A_{n}(k)=d^{\frac{r(r-1)}{2}}\prod_{\mbox{\tiny$\begin{array}{c}
p\mid n\\
p\nmid n/d
\end{array}$}}\prod^r_{i=1}(1-p^{-i})/(1-p^{-1})\ \mathrm{with}\ d=\gcd(k,n).
\end{equation}

Letting $k$ be $1,2,\cdots,n$ in equation (\ref{f58}), we get a system of linear equations:

\begin{equation}\label{f60}
\begin{pmatrix}
 C_n(1) &\cdots & C_{n/l}(1) &\cdots & C_1(1)\\
 C_n(2) &\cdots & C_{n/l}(2) &\cdots & C_1(2)\\
 \vdots& &\vdots& &\vdots\\
 C_n(k) &\cdots & C_{n/l}(k) &\cdots & C_1(k)\\
  \vdots& &\vdots& &\vdots\\
 C_n(n) &\cdots & C_{n/l}(n) &\cdots & C_1(n)
 \end{pmatrix}
\begin{pmatrix}
N_1 \\
\vdots\\
N_l \\
\vdots\\
N_n
\end{pmatrix}
=n^{\frac{r(r-1)}{2}}
\begin{pmatrix}
C_n^r(1)A_{n}(1)\\
C_n^r(2)A_{n}(2)\\
\vdots\\
C_n^r(k)A_{n}(k)\\
\vdots\\
C_n^r(n)A_{n}(n)
\end{pmatrix}
\end{equation}

By (\ref{f55}), the columns of coefficient matrix in (\ref{f60}) are orthognal. Taking the inner product, we get
\begin{equation}\label{f61}
l\cdot n/l\cdot \varphi(n/l)\cdot N_l=n^{\frac{r(r-1)}{2}}\sum_{k=1}^nC_{n/l}(k)C_n^r(k)A_{n}(k)
\end{equation}
by (\ref{f55}) again. Since the summation item in the right hand side of (\ref{f61}) only depends on the $d={\rm gcd}(k,n)$, we get

\begin{equation}\label{f62}
N_l=\frac{1}{\varphi(n/l)}n^{\frac{r(r-1)}{2}-1}\sum_{d\mid n}C_{n/l}(d)C_n^r(d)A_{n}(d)\varphi(n/d)
\end{equation}

Substituting (\ref{f59}) into (\ref{f62}), we get the desired result.
\end{proof}

The Theorem \ref{Thm6} looks complicated for general $n$. So we calculate the special case of $n$ being a prime power. The result seems to be much simpler.

\begin{thm}\label{Thm7} Let $n=p^m$ be a prime power with $m\geqslant 1$. For $\beta\in{\Bbb Z_n}$, let $p^t={\rm gcd}(\beta,p^m)$ with $0\leqslant t\leqslant m$. Then

\begin{equation}\label{f62(1)}
N_{\beta}=
\begin{cases}
p^{m(r^2-1)}\left(\prod_{i=1}^r(1-p^{-i})-(-1)^rp^{-r(r+1)/2}\right),&\text{if~$t=0$},\\
p^{m(r^2-1)}\left(\prod_{i=1}^r(1-p^{-i})+(-1)^rp^{-r(r+1)/2}(p-1)\right),&\text{otherwise}.
\end{cases}
\end{equation}
\end{thm}
\begin{proof} By Theorem \ref{Thm6}, we have
\begin{equation}\label{f63}
\begin{split}
N_{\beta}=&\frac{1}{\varphi(p^{m-t})}p^{m(r(r-1)/2-1)}\sum_{u=0}^mC_{p^{m-t}}(p^{u})C_{p^{m}}^r(p^{u})\varphi(p^{m-u})p^{ur(r-1)/2}\\
&\times\prod_{p\nmid p^{m-u}}\prod^r_{i=1}(1-p^{-i})/(1-p^{-1})
\end{split}
\end{equation}

From (\ref{f54}), we have

\begin{equation}\label{f64}
C_{p^{m}}(p^{u})=
\begin{cases}
\varphi(p^{m}),&u\geqslant m,\\
-p^{m-1},&u= m-1,\\
0,&0\leqslant u< m-1.
\end{cases}
\end{equation}
Therefore only for $u=m, m-1,$ the corresponding summation items in (\ref{f63}) are nonzero, which are
\begin{equation}\label{f65}
\varphi(p^{m-t})\varphi(p^{m})^rp^{mr(r-1)/2}\prod^r_{i=1}(1-p^{-i})/(1-p^{-1})
\end{equation}
and
\begin{equation}\label{f66}
C_{p^{m-t}}(p^{m-1})(-p^{m-1})^r\varphi(p)p^{(m-1)r(r-1)/2}
\end{equation}
respectively. By (\ref{f64}), we have
\begin{equation}\label{f67}
C_{p^{m-t}}(p^{m-1})=
\begin{cases}
-p^{m-1},&\text{if~} t=0,\\
\varphi(p^{m-t}),&\text{if~} 0<t\leqslant m.
\end{cases}
\end{equation}

Substituting (\ref{f67}) into (\ref{f66}), then putting (\ref{f65}) and (\ref{f66}) into (\ref{f63}), we get the desired result.
\end{proof}
\begin{rmk} In Theorem \ref{Thm7}, it is interesting to note that $N_{\beta}$ has only two vales depending $p|{\beta}$ or not. This is clearly true in the case of $m=1$. In fact, Theorem 3.1 of \cite{LH} gives the exact value of $N_{\beta}$ for $m=1$, e.g.
\begin{equation}\label{f68}
N_{0}=p^{r^2-1}\left(\prod_{i=1}^r(1-p^{-i})+(-1)^rp^{-r(r+1)/2}(p-1)\right)
\end{equation}

For general $m$, this has the following explanations. Consider the natural homomorphism:
\begin{equation}\label{f68(1)}
{\rm GL}_r(\mathbb{Z}_{p^m})\xrightarrow{\pi}{\rm GL}_r(\Bbb F_p),~X\mapsto X\ {\rm mod}\ p,
\end{equation}
where $X\in{{\rm GL}_r(\Bbb Z_{p^m})}$. Denote $Y=\pi(X)$. Then clearly,
\begin{equation}\label{f68(2)}
{\rm tr}\ Y \equiv {\rm tr}\ X \pmod p,\ {\rm det}\ Y\equiv {\rm det}\ X\pmod p.
\end{equation}

Fixed $Y\in{{\rm GL}_r(\Bbb F_p)}$, let $X_0$ be any matrix int ${\rm GL}_r(\Bbb Z_{p^m})$ such that $X_0\equiv Y\pmod p$. Then
\begin{equation}\label{f69}
{\pi}^{-1}(Y)=\{X_0+Z|Z\in{p {\rm M}_r(\Bbb Z_{p^m})}\}
\end{equation}
since ${\rm det}\ (X_0+Z)\equiv {\rm det}\ Y\not\equiv 0\pmod p$, where ${\rm M}_r(\Bbb Z_{p^m})$ is the set of $r\times r$ square matrices over $\Bbb Z_{p^m}$.

Now assume $p|{\beta}$. If ${\rm tr}\ X=\beta$, then ${\rm tr}\ \pi(X)=0$ in $\Bbb F_p$.

Therefore
\begin{equation}\label{f70}
\{X\in{{\rm GL}_r(\Bbb Z_{p^m})}|{\rm tr}\ X=\beta\}=\bigcup_{{\rm tr}\ Y=0}\{X\in{{\pi}^{-1}(Y)}|{\rm tr}\ X=\beta\}.
\end{equation}

Fixed $Y$, by (\ref{f69}), we have
\begin{equation}\label{f70(1)}
\{X\in{{\pi}^{-1}(Y)}|{\rm tr}\ X=\beta\}=\{X_0+Z|{\rm tr}\ Z=\beta-{\rm tr}\ X_0, Z\in{p {\rm M}_r(\Bbb Z_{p^m})}\}.
\end{equation}
This implies that
\begin{equation}\label{f71}
|\{X\in{{\pi}^{-1}(Y)}|{\rm tr}\ X=\beta\}|=p^{(m-1)(r^2-1)},
\end{equation}
which does not depend on $\beta$.

Combining (\ref{f68}), (\ref{f70}) and (\ref{f71}), we have, for $p|{\beta}$,
\begin{equation}\label{f71(1)}
N_{\beta}=p^{m(r^2-1)}\left(\prod_{i=1}^r(1-p^{-i})+(-1)^rp^{-r(r+1)/2}(p-1)\right).
\end{equation}

 Clearly, for $p\nmid{\beta}$, $N_{\beta}=N_1$ does not depend on $\beta$, which can be easily computed by combining Lemma \ref{lemma1} and equation \eqref{f71(1)}.

Summing up, this actually gives another proof of Theorem \ref{Thm7} based on Theorem 3.1 of \cite{LH}.
\end{rmk}
From the Chinese Remainder Theorem,
\begin{equation*}\label{f71(2)}
{\rm GL}_r(\mathbb{Z}_n)\simeq {\rm GL}_r(\mathbb{Z}_{p_{1}^{m_1}})\times {\rm GL}_r(\mathbb{Z}_{p_{2}^{m_2}})\times \cdots \times {\rm GL}_r(\mathbb{Z}_{p_{s}^{m_s}})
\end{equation*}
where $n=p_1^{m_1}p_2^{m_2}\cdots p_s^{m_s}$ is the prime factorization of $n$. Therefore
\begin{equation}\label{f72}
|\{X\in{{\rm GL}_r(\Bbb Z_n)|{\rm tr}\ X=\beta\}|=\prod_{i=1}^{s}|\{X\in{{\rm GL}_r(\Bbb Z_{p_{i}^{m_i}})}|{\rm tr}\ X=\beta_i}\}|
\end{equation}
where $\beta_i\equiv \beta\ \pmod{ p_{i}^{m_i}}$.

Combining Theorem \ref{Thm7} and (\ref{f72}), we get a product formula of $N_\beta$.

\begin{thm}\label{Thm8} For $\beta\in{\Bbb Z_n}$, let $l={\rm gcd}(n,\beta)$. Then
\begin{equation}\label{f72(1)}
\begin{split}
N_{\beta}=n^{r^2-1}\prod_{p\mid l}&\left(\prod_{i=1}^r(1-p^{-i})+(-1)^rp^{-r(r+1)/2}(p-1)
\right)
\\
\times&\prod_{\mbox{\tiny$\begin{array}{c}
p\mid n\\
p\nmid l
\end{array}$}}\left(\prod_{i=1}^r(1-p^{-i})-(-1)^rp^{-r(r+1)/2}\right)
\end{split}
\end{equation}
\end{thm}

\end{document}